\documentclass [11pt,reqno]{amsart}
\usepackage {amsmath,amssymb,epsfig,enumerate,verbatim,geometry}
\newif\ifpdf
\ifpdf
  \usepackage[pdftex]{graphicx}
  \usepackage[pdftex]{hyperref}
\else
  \usepackage{graphicx}
\fi

\newcommand{\C}{\mathbf{C}}

\renewcommand{\P}{\mathbf{P}}

\newcommand{\fm}{\mathfrak{m}}

\newcommand{\cA}{\mathcal{A}}

\newcommand{\cD}{\mathcal{D}}

\newcommand{\cO}{\mathcal{O}}

\newcommand{\bvP}{{\check{\P}}}
\newcommand{\vB}{{\check{B}}}
\newcommand{\vC}{{\check{C}}}

\newcommand{\vdelta}{{\check{\delta}}}
\newcommand{\vpsi}{{\check{\psi}}}

\newcommand{\eg}{{\rm e.g.\ }} 
\newcommand{\ie}{{\rm i.e.\ }}

\renewcommand{\=}{:=}

\DeclareMathOperator{\reg}{reg}
\DeclareMathOperator{\sing}{sing}

\numberwithin{equation}{section}       

\newtheorem{prop} {Proposition} [section]

\newtheorem{lem}[prop] {Lemma}

\newtheorem{prop-def}[prop]{Proposition-Definition}

\newtheorem*{thmA}{Theorem A} 
\newtheorem*{thmB}{Theorem B} 

\theoremstyle{remark}
\newtheorem{rmk}[prop]{Remark}

\newtheorem*{ackn}{Acknowledgment} 

\title{Algebraic webs invariant under endomorphisms}
\date{\today}

\author{Marius Dabija
  \and
  Mattias Jonsson}
\address{Dept of Mathematics\\
  University of Michigan\\
  Ann Arbor, MI 48109-1109\\
  USA}

\email{mattiasj@umich.edu}
\thanks{Marius Dabija passed away on June 22, 2003.} 

\subjclass[2000]{Primary: 32H50, Secondary: 14C21}

\begin{document}

\begin{abstract}
  We classify noninvertible, holomorphic selfmaps of the projective 
  plane that preserve an algebraic web. In doing so, we obtain
  interesting examples of critically finite maps.
\end{abstract}

\maketitle

%
%
%
%
\section*{Introduction}
In this paper we classify holomorphic
selfmaps of the complex projective plane 
$\P^2$ that are integrable in the quite 
specific sense that they preserve an algebraic web. 

Recall that an \emph{algebraic web} is 
given by a reduced curve $C\subset\bvP^2$, where 
$\bvP^2$ is the dual projective plane consisting of 
lines in $\P^2$. We say that the web is irreducible
if $C$ is an irreducible curve. 
The web is \emph{invariant} for
a holomorphic mapping $f:\P^2\to\P^2$ if every line
in $\P^2$ belonging to $C$ is mapped to another such line.
See Sections~\ref{sec:webs} and~\ref{sec:invariance} for details.
We will assume $f$ is noninvertible.
\begin{thmA}
  If $C$ is irreducible, it is of one of the following types:
  \begin{itemize}
  \item[(i)]
    a line;
  \item[(ii)]
    a smooth conic;
  \item[(iii)]
    a smooth cubic; 
  \item[(iv)]
    a nodal cubic.
  \end{itemize}
\end{thmA}
The maps in~(iii) and~(iv) are always critically finite.
\begin{thmB}
  If $C$ is reducible, it is of one of the following types:
  \begin{itemize}
  \item[(i)]
    the union of two lines;
  \item[(ii)]
    the union of three lines in general position;
  \item[(iii)]
    the union of a conic and a line in general position.
  \end{itemize}
\end{thmB}

The maps appearing in Theorem~A and~B are quite rare among 
all holomorphic selfmaps, but nevertheless interesting. 
Indeed, they provide concrete examples of critically 
finite mappings. See~\cite{FS1,FS2,Uedaconic,Uedacritfin1,Uedacritfin2,critfin,Rong,Koch}
for examples and dynamics of critically finite maps,
and~\cite{Sibony} for a survey of iterations of 
rational maps on projective spaces.

There are several other notions of integrability 
for selfmaps of $\P^2$. 
In~\cite{invpencils} we classified invariant pencils of 
curves. Much more generally,
Favre and Pereira~\cite{FavPer} have classified
invariant foliations for rational maps.
The case of birational maps was studied earlier by 
Cantat and Favre~\cite{CanFav}.
Related work includes the classification of 
totally invariant curves for holomorphic endomorphisms 
of $\P^2$~\cite{FS2,CerveauLinsNeto,dajibathesis,SSU}
and for birational maps of surfaces~\cite{DJS};
see also~\cite{BonDab}.

This note is organized as follows.
After some background in Sections~\ref{sec:webs}
and~\ref{sec:invariance}
we describe in Section~\ref{sec:examples}
the mappings appearing in Theorem~A.
The proof takes place in Section~\ref{sec:proof}
and Section~\ref{sec:reducible} treats
the case of a reducible web.

\begin{ackn}
  We thank the referee for a careful reading of the paper
  and many useful suggestions.
  The second author was supported by the NSF.
  The work was partially completed during a visit to the 
  Mittag-Leffler Institute.
\end{ackn}
%
%
%
%
\section{Algebraic webs and plane geometry}\label{sec:webs}
We start by reviewing some elementary facts of plane geometry.
Let $\P^2$ denote the complex projective plane.  
The support of a divisor $D$ on $\P^2$ is denoted by $|D|$.
Let $\bvP^2$ the dual projective plane, that is, the set
of complex lines in $\P^2$. 
Then $\bvP^2$ is itself isomorphic to the projective plane. 
Let $\cD$ denote the involutive duality between lines
(resp.\ points) in $\P^2$ and points (resp.\ lines) in $\bvP^2$.
Given a point $p\in\P^2$, $\cD p\subset\bvP^2$ is
the line of lines passing through $p$.
Given a line $L\subset\bvP^2$, $\cD L=\bigcap_{l\in L}l$.

Consider a reduced, irreducible curve $B\subset\P^2$
(resp. $B\subset\bvP^2$) of degree $>1$.
The \emph{dual curve} $\vB\subset\bvP^2$ 
(resp $\vB\subset\P^2$)
is the curve of tangents to the local branches to $B$.
If $\psi:A\to B$ is a normalization map,
then $\vpsi:A\to\vB$ defined by $\vpsi(a)=\cD T_aC$,
is also a normalization map. 
Here $T_aC\subset\P^2$ (resp. $T_aC\subset\bvP^2$)
is the tangent line to the irreducible 
curve germ $\psi(A,a)$ at $\psi(a)$. 
The double dual $\check{\check{B}}$ is isomorphic to $B$.

We shall need to compare the degree and singularities of
a curve with those of its dual. To this
end, define the ramification divisor of $\psi:A\to B$ to be
$R_\psi=\sum_{a\in A}(m_\psi(a)-1)a$. Here 
$m_\psi(a)$ is the largest integer $k$ such that
$\psi^*\fm_{B,\psi(a)}\subset\fm_{A,a}^k$,
where $\fm$ denote the maximal ideals.
We then have the following Pl\"ucker-type formula:
\begin{equation}\label{e:eulchar}
  2\deg B-\deg\vB-\deg R_\psi
  =\chi(A)
  =2\deg\vB-\deg B-\deg R_\vpsi,
\end{equation}
where $\chi(A)$ is the topological Euler characteristic of $A$.
This is proved using the Riemann-Hurwitz formula as 
in~\cite[pp.277-280]{GH}; see also~\cite[p.289]{Homma}.

A \emph{web} $W$ on $\P^2$ of degree $\delta$ 
is locally defined by an unordered set of $\delta$ 
holomorphic (possibly singular) foliations. 
In particular, through a general point $p\in\P^2$
passes exactly $\delta$ leaves, and these intersect 
transversely at $p$. Globally, the leaves may exhibit 
complicated behavior and even be dense in $\P^2$.
See \eg~\cite{Per,Pir1,Web} for general facts on webs.

We shall only consider the particular
case of an \emph{algebraic web} on $\P^2$. 
By definition, this is a web $W=W_C$
given by a reduced curve 
$C\subset\bvP^2$ of degree $\delta>1$. 
The leaves of $W_C$ are exactly the lines in 
$\P^2$ corresponding to the points on $C$.
Through a generic point $p\in\P^2$ passes exactly
$\delta$ distinct lines of the web. 
See Figure~\ref{F1} for a picture of the algebraic 
web associated to a conic $C$.

Assume that $C$ is irreducible and $\delta=\deg C>1$.
The normalization map $\psi:A\to C$ then 
induces a symmetric rational map
\begin{equation*}
  \pi:A\times A\dashrightarrow\P^2
\end{equation*}
defined by $\pi(a_1,a_2)=\cD L(\psi(a_1),\psi(a_2))$,
where $L(c_1,c_2)\subset\bvP^2$ is the line passing through 
(distinct) points $c_1,c_2\in\bvP^2$. 
Let us record some facts that are easily established 
by direct computation.
The indeterminacy locus $I_\pi$ of $\pi$ is 
exactly the set of pairs
$(a,a)$ with $a\in R_\psi$ and pairs $(a,b)$ with 
$a\ne b$ but $\psi(a)=\psi(b)$. In particular,
$C$ is smooth if and only if $\pi$ is holomorphic.
We have $\pi(\Delta)=\vC$, where $\Delta\subset A\times A$ 
denotes the diagonal and $\vC\subset\P^2$ the dual curve.
More precisely, $\pi(a,a)=\vpsi(a)$.
Note that $\pi$ has topological degree $\delta(\delta-1)$.
Further, $\pi$ is biholomorphic outside
$\Delta\cup I_\pi$.
\begin{figure}[ht]
  \begin{center}
    \includegraphics[width=0.5\textwidth]{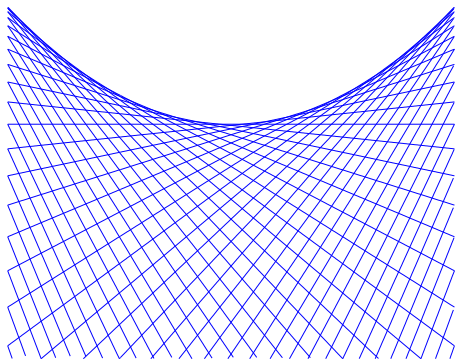}
  \end{center}
  \caption{The algebraic web associated to a conic.}\label{F1}
\end{figure}
%
%
%
%
\section{Selfmaps of curves and of the plane}\label{sec:selfmaps}
Consider a smooth algebraic curve $A$ of genus $g$ 
and a surjective holomorphic map $\phi:A\to A$
of topological degree $d>1$.
The canonical divisor class $K_A$ has degree $2g-2$.
The Riemann-Hurwitz formula asserts 
$K_A=\phi^*K_A+R_\phi$, where $R_\phi$ is the ramification
divisor. Taking degrees, we find
$0\le\deg R_\phi=(d-1)(2-2g)$.  As $d>1$, we have $g=0$ or $g=1$,
that is, $A$ is a rational or elliptic curve.

A subset $E\subset A$ is \emph{totally invariant} 
if $\phi^{-1}(E)\subset E$.  For finite subsets $E$, 
this in fact implies $\phi^{-1}(E)=E=\phi(E)$. 
When $A=\C/\Lambda$ is an
elliptic curve, $\phi$ lifts to an affine map $\tilde\phi:\C\to\C$
and one easily sees that there is no finite totally invariant set.
When $A$ is rational, it is not hard to prove that 
a totally invariant set contains at most two points, 
see \eg~\cite[Theorem~1.5, p.56]{CG}.

A holomorphic map $f:\P^2\to\P^2$ can be written in 
homogeneous coordinates as $f[x:y:z]=[P(x,y,z):Q(x,y,z):R(x,y,z)]$,
where $P$, $Q$ and $R$ are homogeneous polynomials on $\C^3$
of the same degree $d\ge1$, and $\{P=Q=R=0\}=\{0\}$.
The number $d$ is the \emph{algebraic degree} of $f$;
we shall assume $d>1$.
The topological degree of $f$ is $d^2$.
The ramification divisor $R_f$ of $f$ has degree $3(d-1)$. 
It is known~\cite{FS2,CerveauLinsNeto,dajibathesis,SSU} 
that if $C\subset\P^2$ is a (reduced, but possibly reducible)
curve such that $f^{-1}(C)\subset C$, then 
$f^{-1}(C)=C=f(C)$ and $C$ is the union of at most
three lines. Any such line $L\subset C$ occurs with
multiplicity $d-1$ in $R_f$.
%
%
%
%
\section{Invariant webs}\label{sec:invariance}
Consider the algebraic web $W_C$ associated to a curve 
$C\subset\bvP^2$ of degree $\delta$. Assume
for now that $C$ is irreducible. 

Let $f:\P^2\to\P^2$ be a holomorphic mapping of algebraic 
degree $d\ge2$. We say that the web is \emph{invariant} under $f$
if the image under $f$ of any line in the web is again a
line in the web. There is then an induced selfmap 
$g:C\to C$ defined by $g=\cD\circ f\circ\cD$.
The web is \emph{totally invariant}
if the preimage of any line in the web is a union of lines 
in the web.
\begin{prop}
  A web is invariant if and only if it is totally invariant. Moreover,
  if the web is invariant, then:
  \begin{itemize}
  \item[(i)]
    the induced map $g:C\to C$ is regular, 
    of topological degree $d$;
  \item[(ii)]
    for any $c\in C$, we have $f_*\cD c= d\cD g(c)$
    as divisors on $\P^2$.
  \end{itemize}
\end{prop}
\begin{proof}
  Clearly, the web is invariant if it is totally invariant.
  Also,~(ii) is clear, since $f_*$ multiplies the 
  degree of any effective divisor by $d$.

  For $k\ge1$, let $\cA_k\simeq\P^{\frac{k(k+3)}{2}}$ 
  denote the space of effective 
  divisors of degree $k$ on $\P^2$. 
  Then $\cA_1=\bvP^2$.
  Let $\rho_d:\cA_1\to\cA_d$ be the Veronese map
  given by multiplication by $d$:
  $\rho_d(L)=dL$ and $f_*:\cA_1\to\cA_d$ the pushforward
  map induced by $f$.
  Then $f_*$ is holomorphic, of topological degree $d^2$,
  and $\rho_d$ is a holomorphic embedding.
  Suppose the web associated to $C$ is invariant.
  Then $f_*=\sigma_d\circ g$, so $g$ is holomorphic.
  Now $(f_*)^*\cO_{\cA_d}(1)=d^2\cO_{\cA_1}(1)$
  and $\rho_d^*\cO_{\cA_d}(1)\simeq d\cO_{\cA_1}(1)$.
  Restricting to $C$ implies
  $g^*\cO_C(1)=d\cO_C(1)$, so $g$ has topological degree $d$.
  In particular, the preimage of every line in the web
  is the union of (at most) $d$ lines in the web, so
  the web is totally invariant.
\end{proof}
The induced selfmap $g:C\to C$ preserves collinearity:
if $c_1,c_2,c_3\in C$ are collinear in $\bvP^2$, then so
are $g(c_1),g(c_2),g(c_3)$. Indeed, 
$\bigcap_i\cD g(c_i)=f(\bigcap_i\cD c_i)$.
Conversely, if 
$g:C\to C$ is a surjective holomorphic map preserving
collinearity, then there is a unique holomorphic mapping
$f:\P^2\to\P^2$ satisfying $g=\cD\circ f\circ\cD$.

Clearly, $g:C\to C$ lifts uniquely through the 
normalization map $\psi:A\to C$ to a 
holomorphic selfmap $\phi:A\to A$ of topological degree $d>1$. 
In particular, $A$ is rational or elliptic.
Moreover, $f:\P^2\to\P^2$ lifts 
through the rational map $\pi:A\times A\dashrightarrow\P^2$
to the selfmap 
$A\times A\overset{(\phi,\phi)}\longrightarrow A\times A$. 
This implies $f(\vC)=\vC$.

For the proof of Theorem~A we need to compare the 
ramification divisors $R_f\subset\P^2$ and
$R_\phi\subset A$ of $f$ and $\phi$, respectively.
In general, $R_f$ has degree $3(d-1)$. Since $f$ 
preserves the algebraic web associated to $C\subset\bvP^2$,
we can write $R_f=R_f^C+R_f^\sigma$, where $R_f^C$
is the part of $R_f$ supported on the lines of the web,
and $R_f^\sigma$ is the ``sectional'' part of $R_f$.
\begin{lem}\label{l:fibercritset}
  If $a\in A$ and $\psi(a)\notin\sing C$,
  then the multiplicity of 
  the point $a$ in $R_\phi$ equals the multiplicity of 
  the line $\cD\psi(a)$ in $R^C_f$.
\end{lem}
\begin{proof}[Proof of Lemma~\ref{l:fibercritset}]
  If $a,b\in A$,
  $\psi(a),\psi(b)\not\in\sing C$ and $\psi(a)\ne\psi(b)$,
  then $\pi:A\times A\dashrightarrow\P^2$ is a local
  biholomorphism at $(a,b)$. 
  Picking $b$ generic gives the result.
\end{proof}
%
%
%
%
\section{Examples}\label{sec:examples}
We now go through the examples appearing in 
Theorem~A and briefly discuss their dynamics.
%
%
\subsection{A line}\label{sec:line}
The case when $C$ is a line corresponds to a pencil of 
lines through a point $p\in\P^2$. In other words, $C=\cD p$.
We have $\deg R_f^C=2(d-1)$ and $\deg R_f^\sigma=(d-1)$.

Pick homogeneous coordinates $[x:y:z]$ on $\P^2$
such that $p=[0:0:1]$. Then $f$ preserves $C$ if and only if
it takes the form $f[x:y:z]=[P(x,y):Q(x,y):R(x,y,z)]$.

Holomorphic selfmaps of $\bvP^2$ preserving a pencil of
curves were classified in~\cite{invpencils}. 
Their dynamics is studied in~\cite{polyskew,fibered}.
%
%
\subsection{A conic}\label{sec:conic}
Suppose $C\subset\bvP^2$ is a smooth conic so that 
$A\simeq C\simeq\P^1$. Pick any holomorphic selfmap
$\phi:\P^1\to\P^1$ of degree $d>1$. The map 
$\pi:\P^1\times\P^1\to\P^2$ is then holomorphic,
of topological degree 2, and 
$\P^1\times \P^1\overset{(\phi,\phi)}\longrightarrow \P^1\times \P^1$
induces a holomorphic selfmap $f:\P^2\to\P^2$
of algebraic degree $d$.
Any holomorphic selfmap $f$ of $\P^2$ preserving
a conic $C\subset\bvP^2$ is of this form.
We have $\deg R_f^C=2(d-1)$ and $\deg R_f^\sigma=(d-1)$.
The dual curve $\vC\subset\P^2$ is an invariant 
smooth conic and $f^*\vC=\vC+2|R^\sigma_f|$,
see Lemma~\ref{l:sectcritset}.
Selfmaps as above were first introduced in a dynamic
setting by Ueda~\cite{Uedaconic}. They can be used to
provide simple examples of critically finite maps, in particular 
maps whose Julia set is all of $\P^2$~\cite[Proposition 4.1]{Uedaconic}.
%
%
\subsection{A smooth cubic}\label{sec:smoothcubic}
Let $C\subset\bvP^2$ be a smooth cubic. 
Then $C\simeq\C/\Lambda$ is an elliptic curve and 
admits a group law:
three points in $C$ are collinear if and only if
their sum is zero.
Choose the origin of the group law $(C,+,0)$ at 
any flex of $C$ and consider any 
holomorphic selfmap $g:C\to C$.
Then $g$ preserves collinearity if and only if
its translation factor is a flex of $C$.
In this case, $g$ induces a selfmap
$f:\P^2\to\P^2$ preserving the web associated to $C$,
and any such $f$ is of this form.

We have $R_f^C=\emptyset$ and $\deg R_f^\sigma=3(d-1)$.
The dual curve $\vC\subset\P^2$ is an invariant 
sextic with nine cusps and
$f^*\vC=\vC+2|R^\sigma_f|$, 
see Lemma~\ref{l:sectcritset}.

In particular, $f$ is always critically finite,
$f(|R_f|)=\vC=f(\vC)$.
It follows from~\cite[Theorem~5.9]{Uedacritfin1} that
the Julia set of $f$ is all of $\P^2$.
%
%
\subsection{A nodal cubic}\label{sec:nodalcubic}
Let $C\subset\bvP^2$ be a nodal cubic, 
with the node at $c_*$. 
There is a geometrically defined multiplicative group 
law on $C^*\=C\setminus\{c^*\}$, given as follows:
$c_1c_2c_3=e$ in the group 
if and only if $c_1$, $c_2$ and $c_3$ are
collinear in $\bvP^2$. Here the
unit element $e$ can be chosen as any of the
three flexes of $C$.
Concretely, the unique normalization map
$\P^1\simeq A\overset{\psi}\rightarrow C$ such that 
$\psi(0)=\psi(\infty)=c_*$ and $\psi(1)=e$
restricts to a group homomorphism $\psi:\C^*\to C^*$.
We then see that $g:C\to C$ preserves collinearity if and only if
$\phi:A\to A$ takes the form $\phi(a)=\tau a^{\pm d}$,
where $\tau^3=1$. In fact, by changing the flex 
representing $e$, we obtain $\tau=1$.

We have $\deg R_f^C=(d-1)$ and $\deg R_f^\sigma=2(d-1)$.
The dual curve $\vC\subset\P^2$ is an invariant 
quartic with three cusps and one bitangent.
We have $f^*\vC=\vC+2|R^\sigma_f|$, 
see Lemma~\ref{l:sectcritset}.
In particular, $f$ is critically finite.

In suitable coordinates on $\bvP^2$ and $\P^2$, we have
$C=\{u^3+v^3=uvw\}$, 
$\psi(a)=[-a^2:a:a^3-1]$,
$\vpsi(a)=[2a^3+1:2a+a^4:a^2]$
and $\pi(a,b)=[ab(a+b)+1:a+b+a^2b^2:ab]$.
The selfmap of $\P^2$ 
associated to $\phi(a)=a^d$ is then a polynomial
mapping $f_d:\C^2\to\C^2$ of the form 
$f_d(x,y)=(A_d(x,y,1),A_d(y,x,1))$, where
$A_d(x+y+z,xy+yz+zx,xyz)=x^d+y^d+z^d$.
For example, $f_2(x,y)=(x^2-2y,y^2-2x)$.
%
%
%
%
\section{Proof of Theorem~A}\label{sec:proof}
Without loss of generality, 
assume that $\delta\=\deg C\ge3$ and that $C$ is
not a smooth cubic. We shall prove, in a self-contained way,
that $C$ is a nodal cubic. An alternative approach,
suggested by the referee, is to use a result in web theory
to first show that $\deg C\le 3$: see Remark~\ref{rmk:referee}.

We need two results, the proofs of which are given below.
Let $m_c(C)$ denote the multiplicity of the curve $C$ at
a point $c$.
\begin{lem}\label{l:singtotinv}
  The singular locus $\sing C$ is totally invariant for $g:C\to C$.
  Moreover, if $c\in\sing C$, then $m_c(C)=\delta-1$.
  As a consequence, the set $A_s\=\psi^{-1}(\sing C)$ is
  totally invariant for $\phi:A\to A$.
\end{lem}
For the second result, recall the notation $R_f^C$ and $R_f^\sigma$ 
for the fiber and sectional parts of the ramification locus of
$f$, respectively.
\begin{lem}\label{l:sectcritset}
  We have $\deg\vC\le\frac{2}{d-1}\deg R_f^\sigma$ with 
  equality if and only if $f^*\vC=\vC+2|R_f^\sigma|$ as divisors.
\end{lem}

Now let us prove Theorem~A.
The map $\phi:A\to A$ has topological 
degree $d>1$, so $A$ has to be a rational or elliptic curve.
Since we have assumed $\delta\ge 3$, $C$ is singular.
Lemma~\ref{l:singtotinv} implies that
$A_s=\psi^{-1}(\sing C)$ is a totally invariant set
for $\phi:A\to A$. This is impossible if $A$ is elliptic,
so $C$ must be rational. Moreover, $A_s$
consists of one or two points. Each such point 
corresponds to a line in $\P^2$ that is totally invariant for
$f$, and hence contributes to $R_f^C$ as a divisor of degree
$d-1$.

\textbf{Case 1}: $\#A_s=1$. The point 
in $A_s$ contributes a line of multiplicity $d-1$
to $R_f^C$. By Lemma~\ref{l:fibercritset},
the critical points of $\phi$ in $A\setminus A_s$
contribute lines of total degree $d-1$ in $R_f^C$.
Thus $\deg R_f^\sigma=d-1$.
Lemma~\ref{l:sectcritset} gives $\vdelta\le2$, so
that $\vC$, and hence $C$, is a conic.
This contradicts the assumption $\delta\ge 3$.

\textbf{Case 2}: $\#A_s=2$.
Then $\phi:A\to A$ has no critical points outside $A_s$,
so $\deg R_f^\sigma=d-1$ or $2(d-1)$, depending on
whether $\#\sing C=2$ or $\#\sing C=1$. 
If $\deg R_f^\sigma=d-1$, Lemma~\ref{l:sectcritset}
implies $\vdelta\le2$, so that $\vC$, and hence $C$
is a smooth conic, contradicting $\delta\ge3$.
Hence suppose $\#\sing C=1$ and $\deg R_f^\sigma=2(d-1)$.
Write $A_s=\{a,b\}=|R_\psi|$. Then $\psi(a)=\psi(b)=\sing C$.
By Lemma~\ref{l:singtotinv},
$\deg R_\psi=(m_\psi(a)-1)+(m_\psi(b)-1)=\delta-3$.
It follows from~\eqref{e:eulchar} applied to $B=C$
that $\delta=\vdelta-1$. 
Now Lemma~\ref{l:sectcritset} shows that $\vdelta\le 4$.
We have assumed $\delta\ge3$, hence $\delta=3$.
Since $\#\sing C=1$ and $\# A_s=2$, $C$ is a nodal cubic.

\begin{proof}[Proof of Lemma~\ref{l:singtotinv}]
  We may assume that $\delta\ge 3$ or else the irreducible
  curve $C$ is smooth, and there is nothing to prove.
  
  Note that for any $c\in C$, a generic line
  $l\subset\bvP^2$ through $c$ intersects $C$ 
  in exactly $\delta+1-m_c(C)$ points.
  Dually, through a generic point on $\cD c\subset\P^2$
  passes exactly $\delta+1-m_c(C)$ distinct lines of
  the web. In particular, $c\in\reg C$ if and only if
  a generic point on the line $\cD c\subset\P^2$ 
  belongs to $\delta$ distinct lines of the web.

  Now consider $c\in C$ and $c'\=g(c)\in C$.
  Let $p\in\cD c\subset\P^2$ be a generic point and
  write $p'\=f(p)$.
  Assume that $c'\in\sing C$.
  We will show that $c\in\sing C$.
  This will prove that $\sing C$ is totally invariant.

  First assume $\cD c\not\subset|R_f|$. 
  Then $f$ is a local biholomorphism at $p$,
  so $p$ and $p'$ belong to the same number of 
  lines of the web.
  Thus $c\in\sing C$.

  Now assume $\cD c\subset|R_f|$. 
  Then the kernel of the differential 
  $Df_p$ is one-dimensional.
  Thus there is at most one line $L$
  in $\P^2$ through $p$ such that the curve germ
  $f(L,p)$ is transverse to the line $\cD c'$ at $p'$.
  This implies that $\delta+1-m_c(C)\le 2$,
  that is, $m_c(C)\ge\delta-1$. 
  The reverse inequality always holds
  (since $C$ is irreducible and not a line), 
  and so $m_c(C)=\delta-1$. 
  In particular $c\in\sing C$.
 
  Thus $\sing C$ is totally invariant.
  This implies that $\cD c\subset|R_f|$ for every 
  $c\in \sing C$. The above argument
  then shows $m_c(C)=\delta-1$ for every $c\in\sing C$.
\end{proof}
\begin{proof}[Proof of Lemma~\ref{l:sectcritset}]
  The dual curve $\vC$ is the set of points in $\P^2$ 
  belonging to $<\delta$ lines of the web.
  It is therefore clear that 
  $f^{-1}\vC\subset\vC\cup|R_f|$.
  In fact, we have
  $f^{-1}\vC\subset\vC\cup|R_f^\sigma|$ since no line in the web is
  mapped into the irreducible curve $\vC$.

  Write 
  \begin{equation*}
    f^*\vC=a\vC+\sum_jm_j X_j,
  \end{equation*}
  where $a\ge 0$, $m_j>1$ and $X_j$ are 
  irreducible components of $|R_f^\sigma|$.
  Write $\lambda_j=\deg X_j$.
  Then $(d-a)\vdelta=\sum m_j\lambda_j$ and 
  \begin{multline*}
    \deg R_f^\sigma
    \ge\max\{a-1,0\}\vdelta+\sum(m_j-1)\lambda_j\\
    \ge\max\{a-1,0\}\vdelta+\frac12\sum m_j\lambda_j\\
    =\frac12(2\max\{a-1,0\}\vdelta+(d-a)\vdelta)
    \ge\frac12(d-1)\vdelta.
  \end{multline*}  
  Equality holds if and only if $a=1$, $m_j=2$ for all $j$
  and $|R_f^\sigma|=\sum X_j$. 
\end{proof}
%
%
%
%
\section{The reducible case}\label{sec:reducible}
We now prove Theorem~B.
Let $C_j$ be the irreducible components of $C$
and write $\delta_j=\deg C_j$.
Replacing $f$ by an iterate, we may assume
that the web associated to each $C_j$ is (totally)
invariant for $f$.
Write $R_f=R_f^{C_j}+R_f^{\sigma_j}$ for any $j$.

First assume $\delta_j>1$ for some $j$, say $j=1$.
From the analysis above we have $f(|R_f^{\sigma_1}|)=\vC_1$.
For $i>1$, this implies $\delta_i=1$ and 
$R_f^{\sigma_1}=R_f^{\sigma_i}$, hence also
$R_f^{C_1}=R_f^{C_i}$.
Taking degrees and consulting Section~\ref{sec:examples}
we see that $C_1$ must be a conic. It has to intersect each
line $C_i$, $i>1$,  in two distinct points $c_1$, $c_2$.
Then $R_f^{C_1}=(d-1)(\cD c_1+\cD c_2)$.
In particular, the line $C_i$ is unique.
Thus $C=C_1\cup C_2$, where $C_1$ is a conic, 
$C_2$ is a line and $C_1\cap C_2=\{c_1,c_2\}$.
As $f$ preserves the web $W_{C_1}$,
it comes from a selfmap of $C_1\simeq\P^1$ 
for which $c_1$ and $c_2$
are totally invariant. Conversely, if $g$ is such a 
selfmap, the associated map $f:\P^2\to\P^2$ leaves the 
lines $\cD c_1$ and $\cD c_2$ totally invariant.
One can then check that $f$ preserves the linear
pencil of lines through $\cD c_1\cap \cD c_2$, that is,
the web $W_{C_1}$.

Now suppose all the irreducible components of $C$
are lines. We cannot have three concurrent lines, 
as then $f$ would admit a totally invariant line $l$
such that the restriction of $f$ to $l$ would
have three totally invariant points.
We also cannot have four lines, as $f$ then would admit 
five totally invariant lines.

If $C$ is a union of two lines, $f$ can be written in 
suitable coordinates as a polynomial product map
$f(x,y)=(p(x),q(y))$, where $\deg p=\deg q=d$.
If $C$ is a union of three lines, then
$f(x,y)=(x^d,y^d)$.

\begin{rmk}\label{rmk:referee}
  As the referee points out, one can use known results 
  in web geometry to
  directly show that $\deg C\le 3$ in Theorems~A and~B. 
  Indeed, it is known that if
  $(W,0)$ and $(W',0)$ are germs of linear webs of
  degree $\ge 4$ on $\C^2$
  (\ie the leaves are lines and through a general point passes at least
  four leaves) 
  and $\Phi:(\C^2,0)\to(\C^2,0)$ is a local biholomorphism
  mapping leaves to leaves, then $\Phi$ extends to an automorphism
  of $\P^2\supset\C^2$. 
  If $\deg C\ge4$, one can get a contradiction by taking $\Phi$ as
  the germ of $f:\P^2\to\P^2$ at a generic point.
  The result above goes back to the Hamburg school of web geometry:
  see~\cite[Section 42]{BlaBol} or~\cite[Corollaire~2, p.535]{Henaut}
  for a precise modern reference.
  See also~\cite{Pir2} for related results.
\end{rmk}

%
%
%
%

\end{document}